\documentclass[a4paper,12pt]{article}
\usepackage{authblk}
\usepackage{a4wide}
\usepackage{amsfonts}
\usepackage{epsfig}
\usepackage{makeidx}
\usepackage{amsmath,amssymb,latexsym}
\usepackage{enumerate}
\usepackage{graphicx}

\newtheorem{thm}{Theorem}

\newtheorem{prop}[thm]{Proposition}
\newtheorem{lem}[thm]{Lemma}

\newtheorem{claim}[thm]{Claim}

\newenvironment{proof}{\textit{Proof.}}{\hfill $\Box$ \\}

\newcommand\blfootnote[1]{%
  \begingroup
  \renewcommand\thefootnote{}\footnote{#1}%
  \addtocounter{footnote}{-1}%
  \endgroup
}

\begin{document}

\title{On the bend number of circular-arc graphs as edge intersection graphs of paths on a grid}

\author[a]{Liliana Alc\'on}
\author[b,h]{Flavia Bonomo}
\author[c,d,h]{Guillermo Dur\'an}
\author[a,h]{Marisa Gutierrez}
\author[a,h]{Mar\'\i a P\'\i a Mazzoleni}
\author[e,f]{Bernard Ries}
\author[g]{Mario Valencia-Pabon}

\affil[a]{Dto. de Matem\'atica, FCE-UNLP, La Plata, Argentina}
\affil[b]{Dto. de Computaci\'on FCEN-UBA, Buenos Aires, Argentina}
\affil[c]{Dto. de Matem\'atica e Inst. de C\'alculo FCEN-UBA,
Buenos Aires, Argentina}
\affil[d]{Dto. de Ingenier\'\i a Industrial, FCFM-Univ. de Chile,
Santiago, Chile}
\affil[e]{PSL, Universit\'e Paris-Dauphine, Paris, France}
\affil[f]{CNRS, LAMSADE UMR 7243, Paris, France}
\affil[g]{Universit\'e Paris-13, Sorbonne Paris Cit\'e LIPN, CNRS
UMR7030, Villetaneuse, France. Currently in D\'el\'egation at the
INRIA Nancy - Grand Est, France}
\affil[h]{CONICET}

\date{\textbf{Dedicated to Martin Charles Golumbic \\
on the occasion of his 65th birthday}} \maketitle

\vspace{-1cm}

\blfootnote{E-mail addresses: liliana@mate.unlp.edu.ar;
fbonomo@dc.uba.ar;
 gduran@dm.uba.ar; marisa@mate.unlp.edu.ar; pia@mate.unlp.edu.ar; bernard.ries@dauphine.fr; valencia@lipn.univ-paris13.fr. \\
 This work was partially supported by MathAmSud Project 13MATH-07
(Argentina--Brazil--Chile--France), UBACyT Grant 20020130100808BA,
CONICET PIP 122-01001-00310, 112-200901-00178 and
112-201201-00450CO and ANPCyT PICT 2010-1970 and 2012-1324
(Argentina), FONDECyT Grant 1140787 and Millennium Science
Institute ``Complex Engineering Systems'' (Chile).}


\begin{abstract}
Golumbic, Lipshteyn and Stern \cite{Golumbic-epg} proved that every graph can be represented as the edge intersection graph of paths on a grid (EPG graph), i.e., one can associate with each vertex of the graph a nontrivial path on a rectangular grid such that two vertices are adjacent if and only if the corresponding paths share at least one edge of the grid. For a nonnegative integer $k$, $B_k$-EPG graphs are defined as EPG graphs admitting a model in which each path has at most $k$ bends. Circular-arc graphs are intersection graphs of open arcs of a circle. It is easy to see that every circular-arc graph is a $B_4$-EPG graph, by embedding the circle into a rectangle of the grid. In this paper, we prove that every circular-arc graph is $B_3$-EPG, and that there exist circular-arc graphs which are not $B_2$-EPG. If we restrict ourselves to rectangular representations (i.e., the union of the paths used in the model is contained in a rectangle of the grid), we obtain EPR (edge intersection of path in a rectangle) representations. We may define $B_k$-EPR graphs, $k\geq 0$, the same way as $B_k$-EPG graphs. Circular-arc graphs are clearly $B_4$-EPR graphs and we will show that there exist circular-arc graphs that are not $B_3$-EPR graphs. We also show that normal circular-arc graphs are $B_2$-EPR graphs and that there exist normal circular-arc graphs that are not $B_1$-EPR graphs. Finally, we characterize $B_1$-EPR graphs by a family of minimal forbidden induced subgraphs, and show that they form a subclass of normal Helly circular-arc graphs.\\

\end{abstract}


\section{Introduction}

Let $\mathcal{G}$ be a rectangular grid of size $(\ell+1) \times (\ell+1)$. The horizontal grid lines will be referred to as \textit{rows} and denoted by $x_0,x_1,\ldots,x_{\ell}$, and the vertical grid lines will be referred to as \textit{columns} and denoted by $y_0,y_1,\ldots,y_{\ell}$. A grid point lying on row $x$ and column $y$ is referred to as $(x,y)$. Let $\cal P$ be a collection of nontrivial simple paths on $\cal G$. The edge intersection graph of $\cal P$ (denoted by EPG($\mathcal{P}$)) is the graph whose vertices correspond to the paths of $\cal P$ and two vertices are adjacent in EPG($\cal P$) if and only if the corresponding paths in $\cal P$ share at least one edge in $\cal G$. A graph $G$ is called an \textit{edge intersection graph of paths on a grid (EPG graph)} if $G$=EPG($\cal P$) for some $\cal P$. Every graph $G$ satisfies $G$=EPG($\cal P$) for some $\cal P$ on a large enough grid and allowing an arbitrary number of \textit{bends} (turns on a grid point) for each path \cite{Golumbic-epg}. In recent years, the subclasses for which the number of bends of each path is bounded by some integer $k\geq 0$, known as \textit{$B_k$-EPG graphs}, were widely studied \cite{Bernard-epg,Suk-epg,Biedl-epg,Epstein-epg,Golumbic-epg,Knau-epg-npc,Knau-epg-planar}. The \textit{bend number} of a graph $G$ (resp. a graph class $\mathcal{H}$), is the smallest integer $k\geq 0$ such that $G$ (resp. every graph in $\mathcal{H}$) is a $B_k$-EPG graph. We denote by $B_k$-EPG, $k\geq 0$, the class of $B_k$-EPG graphs.

In \cite{Knau-epg-npc}, it was shown that for every integer $k\geq 0$ there exists a graph with bend number $k$, and that recognizing $B_1$-EPG graphs is NP-complete~. The bend number of classical graph classes was investigated as well. In~\cite{Knau-epg-planar}, it was shown that outerplanar graphs are $B_2$-EPG graphs and that planar graphs are $B_4$-EPG graphs. For planar graphs, it is still an open question whether their bend number is equal to 3 or 4. On the other hand, is easy to see that $B_0$-EPG graphs exactly correspond to interval graphs (i.e., intersection graphs of intervals on a line) \cite{Golumbic-epg}. A generalization of interval graphs are circular-arc graphs (i.e., intersection graphs of open arcs on a circle). It is natural to see circular-arc graphs as EPG graphs by identifying the circle with a rectangle of the grid. Hence, circular-arc graphs form a subclass of $B_4$-EPG graphs. This leads to some natural questions. For example, the bend number of circular-arc graphs or the characterization of circular-arc graphs that are $B_k$-EPG graphs, for some $k<4$. One of the main results of this paper is that the bend number of circular-arc graphs is 3.

Another interesting question is how many bends per path are needed
for a circular-arc graph to be represented in a rectangle of the
grid, i.e., in such a way that the union of the paths is contained
in a rectangle of the grid. We call such graphs \textit{edge
intersection graphs of paths on a rectangle} (\textit{EPR
graphs}). It is easy to see that EPR graphs are exactly the
circular-arc graphs. We will study the classes $B_k$-EPR, for
$0\leq k\leq 4$, in which the paths on the grid that represent the
vertices of the graph have at most $k$ bends. As before, we denote
by $B_k$-EPR, $k\geq 0$, the class of $B_k$-EPR graphs. Similar to
the case of EPG graphs, one can define for a circular-arc graph
$G$ the bend number with respect to an EPR representation as the
smallest integer $k$ such that $G$ is a $B_k$-EPR graph. Notice
that CA = EPR = $B_4$-EPR. We strengthen this observation by
showing that the bend number for circular-arc graphs with respect
to EPR representations is 4. Furthermore, we focus on $B_1$-EPR
graphs and $B_2$-EPR graphs ($B_0$-EPR graphs correspond again to
interval graphs), and relate these classes with the class of
normal Helly circular-arc graphs. In summary, we obtain the
following results: we prove that the bend number of normal
circular-arc graphs with respect to EPR representations is 2;
moreover, we characterize $B_1$-EPR graphs by a family of minimal
forbidden induced subgraphs, and show that they are exactly the
normal Helly circular-arc graphs containing no powers of cycles
$C_{4k-1}^k$, with $k \geq 2$, as induced subgraphs.

An extended abstract of a preliminary version of this work was published in the proceedings of LAGOS~2015~\cite{EPG-lagos15}.


\section{Preliminaries}
\label{s:preliminares}

All graphs that we consider in this paper are connected, finite and simple. For all graph theoretical terms and notations not defined here, we refer the reader to \cite{Bondy-Murty}.

We denote by $C_n$, $n\geq 3$, the chordless cycle on $n$ vertices. A graph is called \emph{chordal}, if it does not contain any chordless cycle of length at least four. Given a graph $G$ and an integer $k \geq 0$, the \emph{power graph} $G^k$ has the same vertex set as $G$, two vertices being adjacent in $G^k$ if their distance in $G$ is at most $k$.

Let $G=(V,E)$ be a graph and let $X\subseteq V$. We denote by $G-X$ the graph obtained from $G$ by deleting all vertices in $X$.

A \textit{clique} (resp. a \textit{stable set}) is a subset of vertices that are pairwise adjacent (resp. nonadjacent). A \emph{thick spider} $S_n$ is the graph whose $2n$ vertices can be partitioned into a clique $K=\{c_1,\dots,c_n\}$ and a stable set $S=\{s_1,\dots,s_n\}$ such that, for $1 \leq i,j, \leq n$, $c_i$ is adjacent to $s_j$ if and only if $i \neq j$. Notice that $S_{n_1}$ is an induced subgraph of $S_{n_2}$ if $n_1 \leq n_2$.

We say that a vertex $v$ \emph{dominates} a vertex $w$ if they are adjacent and every neighbor of $w$ is also a neighbor of $v$.

Let $G$ be a circular-arc graph (\textit{CA graph}). If $\mathcal{C}$ denotes the corresponding circle and $\mathcal{A}$ the corresponding set of open arcs, then $(\mathcal{A},\mathcal{C})$ is called a \emph{circular-arc model} of $G$ \cite{Tuc-char}. A graph $G$ is a \emph{Helly circular-arc graph} (\textit{HCA graph}) if it is a circular-arc graph having a circular-arc model such that any subset of pairwise intersecting arcs has a common point on the circle \cite{Gav-circ-arc}. A circular-arc graph having a circular-arc model without two arcs covering the whole circle is called a \emph{normal circular-arc graph} (\textit{NCA graph}). Circular-arc models that are at the same time normal and Helly are precisely those without three or less arcs covering the whole circle. A graph that admits such a model is called a \emph{normal Helly circular-arc graph} (\textit{NHCA graph}) \cite{L-S-S-nhca}. We will denote by NCA (resp. NHCA) the class of normal (resp. normal Helly) circular-arc graphs.

In \cite{C-G-S-nhca}, the authors present a characterization of
NHCA graphs by a family of minimal forbidden induced subgraphs.
Recent surveys on circular-arc graphs are given in
\cite{D-G-S-survey,L-S-circ-arc-survey}. A very recent
characterization of circular-arc graphs by forbidden structures is
presented in \cite{Hell-circ-arc}.


\section{Circular-arc graphs as EPG graphs} \label{sub:epg}

In this section, we show that the bend number of circular-arc graphs with respect to EPG representations is equal to 3. We start with the following result.

\begin{thm}
\label{t:CA-B3-epg}
Every circular-arc graph is a $B_3$-EPG graph.
\end{thm}

\begin{proof}
Let $G$ be a circular-arc graph and let $(\mathcal{A},\mathcal{C})$ be a circular-arc model of $G$. Without loss of generality, we may assume that the endpoints of the arcs are all distinct and we can number them clockwise in the circle from $1$ to $2n$ (with $n$ being the number of vertices of $G$). We also define a point $0$ in the circle between $2n$ and $1$ (clockwise). The arc $(a,b)$, $1 \leq a,b \leq 2n$, denotes the arc traversing the circle clockwise from endpoint $a$ to endpoint $b$. In particular, an arc $(a,b)$ contains point $0$ of $\mathcal{C}$ if and only if $a > b$. Let $X$ be the set of vertices in $G$ corresponding to arcs containing point $0$ of $\mathcal{C}$. Clearly, these vertices form a clique in $G$. Moreover, $G-X$ is an interval graph that can be represented on a line by taking, for each vertex, the interval $(a,b)$ defined by the endpoints of its corresponding arc, since $a < b$ for vertices in $G-X$. We will construct the following EPG representation of $G$ on a grid. For each vertex in $G-X$ corresponding to an arc $(a,b)$, assign the 3-bends path on the grid whose endpoints are $(x_0,y_b)$ and $(x_b,y_0)$ and whose bend points correspond to the grid points $(x_0,y_a),(x_a,y_a),(x_a,y_0)$. For each vertex of $X$ corresponding to an arc $(a,b)$ (in this case $a > b$), assign the 3-bends path on the grid whose endpoints are $(x_0,y_0)$ and $(x_{2n+1},y_0)$ and whose bend points correspond to the grid points $(x_0,y_b),(x_a,y_b),(x_a,y_0)$. Since all the endpoints of the arcs in $\mathcal{A}$ are different, the edge intersections of the paths are either on row $x_0$ or on column $y_0$ of the grid. Clearly, two paths corresponding to vertices of $G-X$ intersect if and only if the corresponding arcs intersect on $\mathcal{C}$. Two paths corresponding to vertices of $X$ intersect at least on the edge of the grid going from $(0,0)$ to $(0,1)$. The path corresponding to a vertex in $G-X$ with endpoints $(a,b)$ and the path corresponding to a vertex in $X$ with endpoints $(c,d)$ intersect if and only if either $d > a$ or $c < b$, and the same condition holds for the corresponding arcs in $\mathcal{C}$. Thus, we obtain a representation of $G$ as a $B_3$-EPG graph.
\end{proof}

Combining Theorem \ref{t:CA-B3-epg} with the following result
shows that the bend number of circular-arc graphs with respect to
EPG representations is 3.

\begin{prop}\label{p:s40}
The thick spider $S_{40}$ is in $B_3$-EPG $\setminus$ $B_2$-EPG.
\end{prop}

\begin{proof}
Since all thick spiders are circular-arc graphs, it follows from
Theorem \ref{t:CA-B3-epg} that  $S_{40}$ is a $B_3$-EPG graph.
Suppose there exists a $B_2$-EPG representation of $S_{40}$. Let
us consider the path $\mathcal{P}_c$ corresponding to a vertex $c$
of the clique and the paths corresponding to its $39$ neighbors in
the stable set $S$. The path $\mathcal{P}_c$ uses at most three
lines (rows and/or columns) of the grid since it has at most 2
bends. Thus, $\mathcal{P}_c$ intersects at least $13$ paths,
$\mathcal{P}_1,\ldots,\mathcal{P}_{13}$, corresponding to 13 of
its neighbors in $S$ on a same line $x$. Without loss of
generality, we may assume that $x$ is a row of the grid. Notice
that, since the paths have at most 2 bends, the edges of each path
on a same row or column form a connected subpath. Consider now the
13 connected subpaths on $x$ corresponding to the paths
$\mathcal{P}_1,\ldots,\mathcal{P}_{13}$ that $\mathcal{P}_c$
intersects on row $x$. Since these paths correspond to vertices of
$S$, they are edge-disjoint and thus their subpaths on $x$ can be
ordered. We may assume then that $P_1,\ldots,P_{13}$, the subpaths
of $\mathcal{P}_1,\ldots,\mathcal{P}_{13}$ on row $x$, are ordered
by index from left to right. Let $s_j$, $j\in \{1,\ldots,39\}$, be
the vertex in $S$ corresponding to the path $\mathcal{P}_7$, i.e.,
$\mathcal{P}_{s_j}=\mathcal{P}_7$. The path $\mathcal{P}_{c_j}$,
corresponding to vertex $c_j$ of the clique that is not adjacent
to $s_j$, cannot intersect the subpaths corresponding to
$P_1,\ldots,P_{13}$ on both sides of $P_7$ on $x$ since it has at
most two bends. Thus, it intersects at least $6$ of them on some
other row or column. So we may assume, without loss of generality,
that it intersects the paths $\mathcal{P}_8,\ldots,
\mathcal{P}_{13}$ on some other row or column. But since these
paths have at most two bends, are edge-disjoint and all use row
$x$, it follows that $\mathcal{P}_{c_j}$ intersects at most 2
paths among $\mathcal{P}_8,\ldots, \mathcal{P}_{13}$ on a column.
Therefore $\mathcal{P}_{c_j}$ necessarily uses two rows and one
column, and it intersects at least 4 paths among
$\mathcal{P}_8,\ldots, \mathcal{P}_{13}$ on a row $x'\neq x$.
Since these 4 paths use both rows $x,x'$ and they are edge
disjoint, it follows that each of them uses a different column.
Also notice that the order of their corresponding subpaths on the
two rows $x,x'$ from left to right must be the same, since they
have at most 2 bends and therefore they cannot swap the order
without intersecting. Let $s_k$, $k\in \{1,\ldots,39\}$, be the
vertex corresponding to one of the two paths of these four whose
subpaths on row $x'$ are located in the middle (i.e., the second
or third subpath). It is now easy to see that, given that fixed
configuration for the four paths, it is impossible for the path
$\mathcal{P}_{c_k}$, corresponding to vertex $c_k$ of the clique
that is not adjacent to $s_k$, to avoid the path
$\mathcal{P}_{s_k}$ while intersecting the paths corresponding to
the remaining three vertices using only two bends. Thus $S_{40}$
does not admit a $B_2$-EPG representation.
\end{proof}

Finding a characterization of CA $\cap$ $B_2$-EPG and CA $\cap$ $B_1$-EPG by minimal forbidden induced subgraphs is left as an open problem.

\section{Circular-arc graphs as EPR graphs} \label{sub:epr}

In this section, we focus on circular-arc graphs that can be represented as edge intersection graphs of paths on a rectangle of the grid, i.e., we restrict ourselves to 2 rows and 2 columns of the grid. Obviously, $CA=B_4$-EPR since we can embed the CA model into a rectangle of the grid in a natural way. A strengthening of this observation is that the bend number of CA graphs with respect to EPR representations is equal to 4. To show this, it is sufficient to prove the following.

\begin{prop}
The thick spider $S_{13}$ is not in $B_3$-EPR.
\end{prop}

\begin{proof}
By contradiction, suppose that $S_{13}$ admits a $B_3$-EPR
representation. Clearly, at most four of the paths corresponding
to vertices in the stable set contain a corner of the rectangle,
since they are pairwise non adjacent. So, from the remaining $9$
paths, representing vertices in the stable set, at least three of
them are intervals completely contained in one side of the
rectangle. Let us denote these paths in order by $\mathcal{P}_i$,
$\mathcal{P}_j$, $\mathcal{P}_k$, representing vertices $s_i$,
$s_j$, $s_k$. The path corresponding to vertex $c_j$ in the clique
has to intersect  $\mathcal{P}_i$, and $\mathcal{P}_k$ avoiding
$\mathcal{P}_j$, so it necessarily needs four bends, a
contradiction.
\end{proof}

We will now consider normal circular-arc graphs and show that their bend number with respect to EPR representations is equal to 2.

\begin{thm}
Every NCA graph is a $B_2$-EPR graph.
\end{thm}

\begin{proof}
Let $(\mathcal{A},\mathcal{C})$ be a NCA model of a circular-arc graph. Without loss of generality, we may assume that the endpoints of the arcs are pairwise different. Let $p$ be a point of $\mathcal{C}$ that is not the endpoint of an arc of $\mathcal{A}$. Since the model is normal, the union of the arcs of $\mathcal{A}$ that contain $p$ does not cover $\mathcal{C}$. Thus, there exists a point $q$ in $\mathcal{C}$ that is not the endpoint of an arc of $\mathcal{A}$ and is not contained in the union of the arcs of $\mathcal{A}$ containing $p$. We can then embed our model on a rectangle of the grid in the following way (arcs will correspond to paths and $\mathcal{C}$ will correspond to the rectangle): two consecutive corners correspond to point $p$ of the circle and the remaining two corners correspond to point $q$ of the circle. In this way, since no arc of $\mathcal{A}$ contains both $p$ and $q$, paths corresponding to arcs containing either $p$ or $q$ have two bends, while paths corresponding to arcs containing neither $p$ nor $q$ have no bend.
\end{proof}

In order to show that the bend number of NCA graphs with respect
to EPR representations is equal to 2, consider the graph $S_3$.
$S_3$ is not in NHCA~\cite{C-G-S-nhca}, so in particular it is not
in $B_1$-EPR (see Theorem~\ref{t:B1-NHCA-c7c}), but it is in
$B_1$-EPG~\cite{Golumbic-epg} and it is also in NCA (see for
example \cite{C-G-S-nhca}).

Let us also show that there exist $B_2$-EPR graphs that are not in
NCA.

\begin{prop}\label{p:s6}
The thick spider $S_{6}$ is in $B_2$-EPR $\setminus$ NCA.
\end{prop}

\begin{proof}
Let $(\mathcal{A},\mathcal{C})$ be a circular-arc model of $S_{6}$. Without loss of generality, we may assume that the disjoint arcs $\mathcal{A}_1,\dots,\mathcal{A}_6$ representing the vertices $s_1,\dots,s_6$ in the stable set are in clockwise order. It is not difficult to see that the arcs representing vertices $c_1$ and $c_4$ cover the circle. Thus, $S_6$ is not in NCA. A $B_2$-EPR representation of $S_{6}$ is given in Figure~\ref{fig:spiders}.
\end{proof}

We will now focus on $B_1$-EPR graphs and show that they are NHCA graphs.

\begin{lem}
\label{t:B1-NHCA}
$B_1$-EPR $\subseteq$ NHCA.
\end{lem}

\begin{proof}
Consider a $B_1$-EPR representation of a graph $G$ and let $\cal P$ be the set of paths corresponding to the vertices of $G$. We will consider the natural bijection between the rectangle $\mathcal{R}$ and a circle $\mathcal{C}$, that maps the paths in $\mathcal{P}$ to open arcs $\mathcal{A}$ of $\mathcal{C}$. Notice that two open arcs intersect if and only if the corresponding paths of $\mathcal{P}$ intersect on an least one edge of the grid. Thus, $(\mathcal{A},\mathcal{C})$ is a circular-arc representation of $G$. Now, since each path has at most one bend and the arcs are open, the union of three (resp. two) arcs of $\mathcal{A}$ contains at most three (resp. two) points of $\mathcal{C}$ corresponding to corners of $\mathcal{R}$. Hence $(\mathcal{A},\mathcal{C})$ is a NHCA model for $G$.
\end{proof}

Next, we will present a family of NHCA graphs that are not in the class of $B_1$-EPG graphs. First we need to introduce some more definitions. It was shown in \cite{Golumbic-epg} that an induced cycle $C_4$ in a graph $G$ corresponds to either a \textit{true pie} or a \textit{false pie} or a \textit{frame} in any $B_1$-EPG representation of $G$ (see Figure~\ref{fig:C4-B1}).  In a true pie or a false pie, the paths representing the vertices of the induced $C_4$ use one common grid point which is defined as the \textit{center of the pie}. A frame is a model of $C_4$ such that each of the four corresponding paths has a bend in one of the four corners of a rectangle of the grid. Some examples of frame models are shown in Figure~\ref{fig:C4-B1}.

\begin{lem}
\label{t:powers}
Powers of cycles $C_{4k-1}^k$, with $k \geq 2$, are not in $B_1$-EPG.
\end{lem}

\begin{proof}
Let $k \geq 2$ and let $G$ be the graph $C_{4k-1}^k$, where the vertices of the cycle are denoted by $v_1, \ldots, v_n$, with $n=4k-1$. Suppose, by contradiction, that G admits a $B_1$-EPG representation. Note that $v_1$, $v_{k+1}$, $v_{2k+1}$ and $v_{3k+1}$ induce a $C_4$ in $G$.

We will consider the possible representations of the $4$-cycle induced by $v_1$, $v_{k+1}$, $v_{2k+1}$ and $v_{3k+1}$. Let us denote by ${\mathcal P}_i$ the path corresponding to $v_i$, for $i=1,\dots,n$. We will inductively show that, for each of the possible representations, the path $\mathcal{P}_{i+1}$ has to share with ${\mathcal P}_i$ a special point $p_i$ and its two incident edges of the grid. The point $p_i$ will be either the bend of ${\mathcal P}_i$ or the center of the pie in case of pie representations.

Suppose first that the cycle is represented by a true pie using row $x$ and column $y$ of the grid. Let $p$ be the intersection point of $x$ and $y$. Vertex $v_2$ is adjacent to $v_1$,
$v_{k+1}$, and $v_{3k+1}$ in $G$. Since its corresponding path ${\mathcal P}_2$ must intersect ${\mathcal P}_1$, ${\mathcal P}_{k+1}$ and ${\mathcal P}_{3k+1}$ and it cannot intersect ${\mathcal P}_{2k+1}$, it is forced to have a bend at $p$ and use the same semi-row and semi-column as ${\mathcal P}_1$. The same argument can then be applied to $v_{k+2}$ and ${\mathcal P}_{k+2}$ with respect to ${\mathcal P}_{k+1}$, to $v_{2k+2}$ and ${\mathcal P}_{2k+2}$ with respect to ${\mathcal P}_{2k+1}$ and to $v_{3k+2}$ and ${\mathcal P}_{3k+2}$ with respect to ${\mathcal P}_{3k+1}$. Considering now the cycle $v_2$, $v_{k+2}$, $v_{2k+2}$ and $v_{3k+2}$, we can repeat the process and, after $k-1$ iterations, we will reach a contradiction because ${\mathcal P}_1$ should not use the same semi-row and semi-column as ${\mathcal P}_{3k+1}$.

Suppose now that the $4$-cycle is represented by a false pie with center $p=(x,y)$. By symmetry, we may assume that ${\mathcal P}_1$ has a bend at $p$ and ${\mathcal P}_{k+1}$ uses edges on row $x$ on both sides of $p$. Clearly, no path with at most one bend uses edges on two different rows or columns. Since $v_2$ is adjacent to $v_{3k+1}$, $v_1$, and $v_{k+1}$, the path $\mathcal{P}_2$ must have a bend at $p$ and use the same semi-row and semi-column as ${\mathcal P}_1$. A similar argument shows that ${\mathcal P}_{k+2}$ must use edges on row $x$ on both sides of $p$, as ${\mathcal P}_{k+1}$. Symmetrically, an analogous situation holds for ${\mathcal P}_{2k+2}$ with respect to ${\mathcal P}_{2k+1}$ and for ${\mathcal P}_{3k+2}$ with respect to ${\mathcal P}_{3k+1}$. As above, we can repeat the process and, after $k-1$ iterations, we will reach again a contradiction.

Finally, suppose that the $4$-cycle is represented by a frame on the rectangle defined by rows $x$ and $x'$ and columns $y$ and $y'$. Suppose that ${\mathcal P}_1$ has edges on $x$ and $y$,
${\mathcal P}_{k+1}$ on $y$ and $x'$, ${\mathcal P}_{2k+1}$ on $x'$ and $y'$, and ${\mathcal P}_{3k+1}$ on $y'$ and $x$. In order to have edge intersections with ${\mathcal P}_{1}$, ${\mathcal P}_{k+1}$, and ${\mathcal P}_{3k+1}$, the only possible row-column combination for the path ${\mathcal P}_2$ is to use edges on $x$ and $y$. Now repeating similar, iterative arguments as previously, we will reach again a contradiction.
\end{proof}

\begin{figure}
\begin{center}\includegraphics[width=\textwidth]{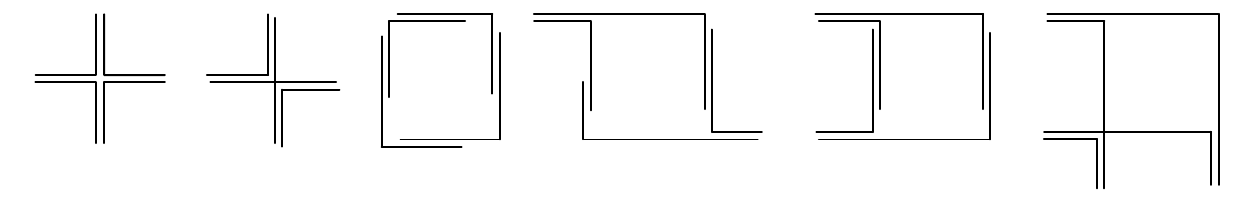}
\caption{From left
to right: a true pie, a false pie and four examples of a frame~\cite{Golumbic-epg}.}\label{fig:C4-B1}
\end{center}
\end{figure}


In the next theorem, we will introduce five equivalent statements
for $B_1$-EPR graphs that are not chordal, two of them using
structural properties, other two in terms of its NHCA models, and
one in terms of forbidden induced subgraphs.

\begin{thm}
\label{t:B1-epr-char}
Let $G=(V,E)$ be a graph which is not chordal. Then the following statements are equivalent:\

\noindent $(i)$ $G\in B_1$-EPR;\

\noindent $(ii)$ $G\in$ NHCA and $G$ contains no $C_{4k-1}^k$, with $k \geq 2$, as induced subgraph;\

\noindent $(iii)$ $G \in$ NHCA and admits a NHCA model $(\mathcal{A}, \mathcal{C})$ with the following property: there are four points of $\mathcal{C}$, different from the endpoints of the arcs of $\mathcal{A}$, such that no arc of $\mathcal{A}$ contains two of these points;\

\noindent $(iii')$ $G \in$ NHCA and in every NHCA model $(\mathcal{A}, \mathcal{C})$ of $G$ there are four points of $\mathcal{C}$, different from the endpoints of the arcs of
$\mathcal{A}$, such that no arc of $\mathcal{A}$ contains two of these points;\

\noindent $(iv)$ $G \in$ NHCA and $G$ has four disjoint connected subgraphs $H_1$, $H_2$, $H_3$, $H_4$, such that $H_1$ and $H_3$ are in different connected components of $G \setminus (V(H_2) \cup V(H_4))$ and $H_2$ and $H_4$ are in different connected components of $G \setminus (V(H_1) \cup V(H_3))$.\

\noindent $(iv')$ $G \in$ NHCA and $G$ has four disjoint complete
subgraphs $H_1$, $H_2$, $H_3$, $H_4$, such that $H_1$ and $H_3$
are in different connected components of $G \setminus (V(H_2) \cup
V(H_4))$ and $H_2$ and $H_4$ are in different connected components
of $G \setminus (V(H_1) \cup V(H_3))$.\
\end{thm}

\begin{proof}
We will prove $(iii) \Rightarrow (iv')$, $(iv) \Rightarrow
(iii')$, $(i) \Rightarrow (ii)$, $(ii) \Rightarrow (iii)$, $(iii)
\Rightarrow (i)$. The implications $(iii') \Rightarrow (iii)$ and
$(iv') \Rightarrow (iv)$ are straightforward.

\

$(iii) \Rightarrow (iv')$: Let $(\mathcal{A}, \mathcal{C})$ be a
circular-arc model of $G$ such that there exist four points $p_1,
\dots, p_4$ (clockwise) of $\mathcal{C}$ satisfying that no arc of
$\mathcal{A}$ contains two of these points. Define $H_i$ as the
subgraph induced by the vertices corresponding to arcs of
$\mathcal{A}$ containing $p_i$, for $i=1,\dots,4$. Since $G$ is
not chordal, the four graphs $H_1, \dots, H_4$ are clearly non
empty complete subgraphs, and since no arc contains two of the
four points, they are disjoint. By the topology of the circle and
the order of the points on it, it follows that every path
connecting a vertex corresponding to an arc containing $p_1$ and a
vertex corresponding to an arc containing $p_3$ necessarily
contains either a vertex corresponding to an arc containing $p_2$
or a vertex corresponding to an arc containing $p_4$. So, $H_1$
and $H_3$ are in different connected components of $G \setminus
(V(H_2) \cup V(H_4))$ and, analogously, $H_2$ and $H_4$ are in
different connected components of $G \setminus (V(H_1) \cup
V(H_3))$.\

\

$(iv) \Rightarrow (iii')$: Let $H_1$, $H_2$, $H_3$, $H_4$ be
disjoint connected subgraphs of $G$ such that $H_1$ and $H_3$ are
in different connected components of $G \setminus (V(H_2) \cup
V(H_4))$ and $H_2$ and $H_4$ are in different connected components
of $G \setminus (V(H_1) \cup V(H_3))$. Let $(\mathcal{A},
\mathcal{C})$ be a NHCA model of $G$, and let $\mathcal{A}^i$ be
the set of arcs corresponding to vertices of $H_i$, for
$i=1,\dots,4$.

The union of the arcs in $\mathcal{A}^i$ induces a connected
sector of $\mathcal{C}$, for each $i=1,\ldots,4$. Moreover, by our
assumptions, $\mathcal{A} \setminus (\mathcal{A}^1 \cup
\mathcal{A}^3)$ induces more than one connected sector on
$\mathcal{C}$, such that arcs in $\mathcal{A}^2$ and in
$\mathcal{A}^4$ are in different connected sectors $A_2$ and
$A_4$. These sectors can be represented as two disjoint arcs of
$\mathcal{C}$ with endpoints $t_2, h_2$, and  $t_4, h_4$
(respectively) in clockwise order (the arc corresponding to $A_i$
is obtained by taking the union of all arcs in belonging to the
same sector as $\mathcal{A}^i$, for $i=2,4$). Define $A_1$ and
$A_3$ analogously, and let us represent them by two disjoint arcs
of $\mathcal{C}$ with $t_1, h_1$, and  $t_3, h_3$ being their
endpoints, respectively. Notice that since the graph is not
chordal, and thus $\mathcal{A}$ covers the circle, either $A_1$
covers $(h_2-\varepsilon,t_4+\varepsilon)$ and  $A_3$ covers
$(h_4-\varepsilon,t_2+\varepsilon)$, or $A_1$ covers
$(t_2-\varepsilon,h_4+\varepsilon)$ and  $A_3$ covers
$(t_4-\varepsilon,h_2+\varepsilon)$, for some $\epsilon > 0$.
Without loss of generality, we may assume that the first case
holds.

Let $p_1$ be a point in the clockwise open arc $(h_2,t_4)$, $p_2$
in $(h_1,t_3)$, $p_3$ in $(h_4,t_2)$, and $p_4$ in $(h_3,t_1)$. No
arc of $\mathcal{A}^1, \mathcal{A}^2, \mathcal{A}^3$ or
$\mathcal{A}^4$ contains two of these points, since these arcs are
contained in $A_1,A_2,A_3$ and $A_4$, respectively. Furthermore,
no arc in  $\mathcal{A}\setminus \bigcup_{i=1}^{4}\mathcal{A}^i$
contains two of these points, because $A_1,A_2,A_3$ and $A_4$ are
connected sectors of either $\mathcal{A} \setminus (\mathcal{A}^1
\cup \mathcal{A}^3)$ or $\mathcal{A} \setminus (\mathcal{A}^2 \cup
\mathcal{A}^4)$.

\

At this point, we know $(iii) \Rightarrow (iii')$.

\

$(i)\Rightarrow (ii)$: This follows directly from Lemmas \ref{t:B1-NHCA} and \ref{t:powers}.\

\

$(ii) \Rightarrow (iii)$: Let $G$ be a minimal counterexample to $(iii)$, and let $(\mathcal{A}, \mathcal{C})$ be an arbitrary NHCA model of $G$. Then for every choice of four points of
$\mathcal{C}$, there is an arc of $\mathcal{A}$ that contains two of these points.

As usual, we will assume that the arcs in $\mathcal{A}$ are open and their endpoints are pairwise distinct. We can number the endpoints clockwise in the circle from $1$ to $2n$ ($n$ being the number of vertices of $G$). For an arc $\mathcal{A}_i \in \mathcal{A}$, its endpoints will be referred to as \emph{tail} and \emph{head} in such a way that $\mathcal{A}_i$ is the open arc traversing $\mathcal{C}$ clockwise from the tail to the head.

\begin{claim}
\label{cl:proper}
No arc of $\mathcal{A}$ is properly contained in another.
\end{claim}

Suppose there is an arc $\mathcal{A}_{i}$ which is properly
contained in an arc $\mathcal{A}_{j}$, $i\neq j$. By minimality,
and since $(iii) \Rightarrow (iii')$, the model $(\mathcal{A}
\setminus \{\mathcal{A}_{i}\}, \mathcal{C})$ admits four points
$p_1, p_2,p_3,p_4$ such that no arc contains two of them. But if
$\mathcal{A}_{j}$ does not contain two of these points, neither
does $\mathcal{A}_{i}$ that is properly contained in
$\mathcal{A}_{j}$. Thus, $(\mathcal{A}, \mathcal{C})$ satisfies
the property as well, a contradiction to our hypothesis.
$\diamondsuit$

\

\begin{claim}
\label{cl:no-dom}
No vertex is dominated by another.
\end{claim}

Suppose that vertex $v$ dominates vertex $w$. Let $\mathcal{A}_{v}$ and $\mathcal{A}_{w}$ be their corresponding arcs in the NHCA model $(\mathcal{A}, \mathcal{C})$. If there is an arc $\mathcal{A}_{z}$, corresponding to a vertex $z$, that intersects $\mathcal{A}_{w}$ only on $\mathcal{A}_{w} \setminus \mathcal{A}_{v}$, then, since $z$ is also adjacent to $v$, $\mathcal{A}_{z}$ intersects $\mathcal{A}_{v}$ only in $\mathcal{A}_{v} \setminus \mathcal{A}_{w}$. But then $\mathcal{A}_{v}$, $\mathcal{A}_{w}$ and $\mathcal{A}_{z}$ cover $\mathcal{C}$, a contradiction. Thus, such an arc does not exist. But then we can replace $\mathcal{A}_{w}$ by $\mathcal{A}_{w} \cap \mathcal{A}_{v}$ obtaining a NHCA model of the same graph with an arc properly contained in another, a contradiction with the previous claim.
$\diamondsuit$

\

It was shown by Golumbic and Hammer~\cite{G-H-circ-arc} that the last claim implies that, when traversing the endpoints of the arcs on the circle clockwise, heads and tails necessarily alternate. Moreover, they proved that $G$ is the $k$-th power of the cycle $C_n$, for some value of $k$. Since $G$ is a counterexample to $(iii)$, it follows that $k \geq 2$.

Let $1, \dots, 2n$ denote the endpoints of the arcs, where odd numbers correspond to tails and even numbers to heads. Thus, every arc is of the form $(2i-1,2i+2k)$, for $i=1,\dots,n$, where the sums are taken modulo $2n$. In particular, every arc properly contains $2k$ of the $2n$ endpoints. Since the model is normal and Helly, $6k < 2n$, otherwise arcs $(1,2k+2)$, $(2k+1,4k+2)$, and $(4k+1,6k+2)$ cover the circle. On the other hand, $8k > 2n$, otherwise points $2$, $2k+2$, $4k+2$, and $6k+2$ would be such that no arc of $\mathcal{A}$ contains two of them.

We will show now that $6k < 2n < 8k$ is also a sufficient
condition for the $k$-th power of the cycle $C_n$ to be a
counterexample to $(iii)$ (not necessarily minimal). It is clear
that $6k < 2n$ ensures that $C^k_n$ is a NHCA graph. Consider a
NHCA model of $C_n^k$. Now, suppose $2n < 8k$ and let  $p_1,
p_2,p_3,p_4$ be four points of $\mathcal{C}$. We may assume that
they correspond to endpoints of arcs, otherwise we can move each
of them to its closest endpoint without creating a new containment
relation between arcs and points. If there are two points at
distance at most $2k-2$, i.e., in the closed interval $[i,
i+2k-2]$ for some $i=1,\ldots,n$, then they are both contained
either in the arc $(i-1,i+2k)$ or in the arc $(i-2,i+2k-1)$,
depending on the parity of $i$. So we may assume now that
$p_1,p_2,p_3,p_4$ are pairwise at distance at least $2k-1$. It
follows from the inequality $2n < 8k$ that at least two pairs of
vertices are at distance exactly $2k-1$ on $\mathcal{C}$, and at
least one of these pairs corresponds to endpoints $i, i+2k-1$ with
$i$ even. Thus, these two points are both contained in the arc
$(i-1,i+2k)$.


Since $k \geq 2$, the inequality $6k < 2n$ implies that $n \geq 7$
thus, by the property above, $C_7^2$ is a minimal counterexample
to $(iii)$. Indeed, $C_n^k$, $k \geq 2$, contains $C_7^2$ as
induced subgraph if and only if $12k < 4n \leq 14k$ (it can be
verified that the arcs $(1,2k+2)$, $(2k+1,4k+2)$, $(4k+1,6k+2)$,
$\dots$, $(12k+1,14k+2)$, where the operations are done modulo
$2n$, induce $C_7^2$).

More in general and inductively, we can prove that  $C_{4t-1}^t$,
with $t \geq 2$, is a minimal NHCA counterexample to $(iii)$ and
that $C_n^k$, $k \geq 2$, contains $C_{4t-1}^t$ as induced
subgraph if and only if $2(4t-5)k < 2(t-1)n$ and $2tn \leq
2(4t-1)k$, or equivalently, $(4t-5)/(t-1) < n/k \leq (4t-1)/t$.

As $(4t-1)/t$ converges to $4$ as $t$ tends to infinity, every
$C_n^k$ with $k \geq 2$ and such that $3 < n/k < 4$ contains a
power of a cycle $C_{4t-1}^t$ as induced subgraph, for some $t
\geq 2$, and this completes the proof.


\

$(iii)\Rightarrow (i)$: Let $(\mathcal{A}, \mathcal{C})$ be a circular-arc model of $G$ such that there exist four points on $\mathcal{C}$ satisfying that no arc of  $\mathcal{A}$ contains
two of these points.

We will place the corners of the rectangle in those four points. Since the arcs in $\mathcal{A}$ are open and we are assuming, without loss of generality, that the endpoints of the arcs are pairwise distinct, we may assume as well that the four points are different from all the arc endpoints. It is easy to see then that we can find a big enough rectangle in the grid such that we can represent the arcs as paths in the grid, maintaining the order of their endpoints, and placing the four corners at the desired points, and this will give us a $B_1$-EPR representation of $G$.
\end{proof}

Now we are able to prove the following.

\begin{thm}
\label{t:B1-NHCA-c7c}
Let $G=(V,E)$ be a graph. Then $G\in B_1$-EPR if and only if $G\in$ NHCA and $G$ has no $C_{4k-1}^k$, $k \geq 2$, as induced subgraph.
\end{thm}

\begin{proof}
One implication follows immediately from Lemmas \ref{t:B1-NHCA}
and \ref{t:powers}, since the class is hereditary. For the
converse, if $G$ is not a chordal graph, then the result follows
from Theorem \ref{t:B1-epr-char}. Suppose now that $G$ is chordal.
It is shown in~\cite{C-G-S-nhca} that a chordal NHCA graph is
indeed an interval graph, thus a $B_0$-EPR graph and, in
particular, a $B_1$-EPR graph.
\end{proof}

Since $C_{4k-1}^k$, $k \geq 2$, is not in $B_1$-EPG (see Lemma~\ref{t:powers}), it follows that $B_1$-EPG $\cap$ NHCA = $B_1$-EPR = NHCA $\cap$ $\{C_{4k-1}^k\}_{k \geq 2}$-free.

It is easy to develop a polynomial-time algorithm to recognize $B_1$-EPR graphs, based on the polynomial-time recognition algorithm of NHCA graphs~\cite{C-G-S-nhca}, that outputs a NHCA model of the graph if there is one, and property $(iii')$ of Theorem \ref{t:B1-epr-char}, since the four points can be chosen from the endpoints of the arcs of the model.

We leave as an open problem the characterization of CA $\cap$
$B_2$-EPR and CA $\cap$ $B_3$-EPR by minimal forbidden induced
subgraphs.

\section{Further results}

The thick spider $S_3$ is one of the minimal forbidden induced
subgraphs for the class NHCA~\cite{C-G-S-nhca}, but all the thick
spiders are CA graphs and, by Theorem~\ref{t:CA-B3-epg}, $B_3$-EPG
graphs. Thick spiders allow us to distinguish classes in the
families $B_k$-EPR ($k \leq 4$), $B_k$-EPG ($k \leq 3$), NHCA and
NCA. In fact, $S_6$ is not NCA but it is in $B_2$-EPR (see
Figure~\ref{fig:spiders} and Proposition~\ref{p:s6}); $S_7$ is in
$B_3$-EPR $\setminus$ $B_2$-EPR (Proposition~\ref{p:s7r});
$S_{40}$ is in $B_3$-EPG $\setminus$ $B_2$-EPG
(Proposition~\ref{p:s40}); $S_{7}$ is in $B_2$-EPG $\setminus$
$B_1$-EPG (Proposition~\ref{p:s7g}).
\\

\begin{figure}
\begin{center}
\includegraphics[width=\textwidth]{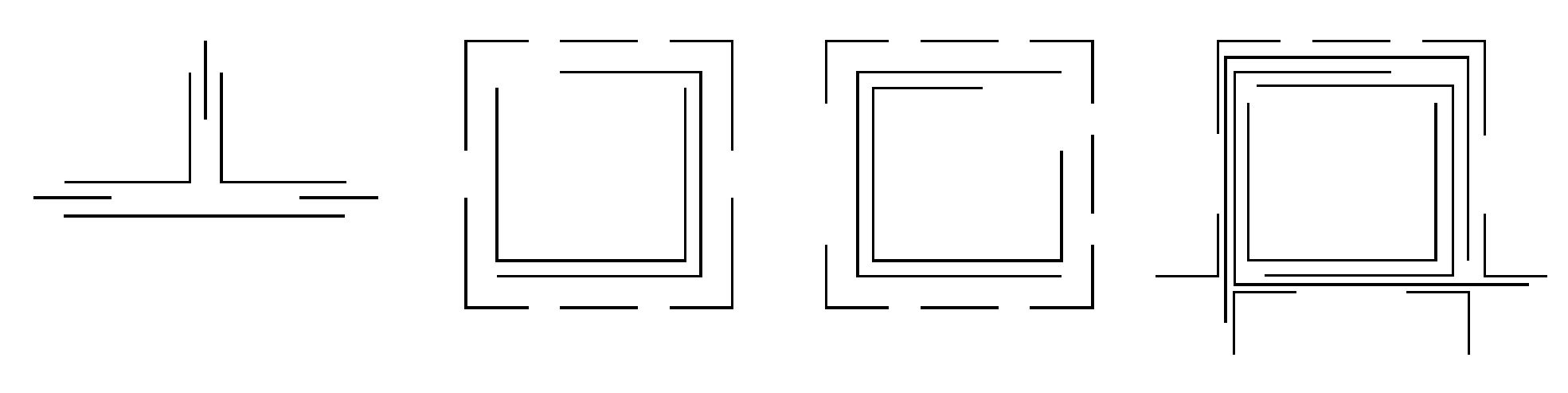}
\caption{From left to right: a $B_1$-EPG representation of $S_3$;
a sketch of a $B_2$-EPR representation of $S_6$ (we draw all the
vertices of the stable set and representative vertices of the
clique; other vertices in the clique can be represented
symmetrically); sketches of $B_3$-EPR and $B_2$-EPG
representations of $S_7$ (again, we draw all the vertices of the
stable set and representative vertices of the clique; other
vertices in the clique can be represented
symmetrically).}\label{fig:spiders}
\end{center}
\end{figure}

We complete the proof of these claims in the following
propositions.

\begin{prop}\label{p:s7r}
The thick spider $S_{7}$ is not a $B_2$-EPR graph.
\end{prop}

\begin{proof}
By contradiction, suppose that $S_7$ admits a $B_2$-EPR
representation. Let us consider only the paths corresponding to
vertices in the stable set. If there are three of them whose union
is contained in two adjacent sides of the rectangle (or in one
side of the rectangle), let us say in order $\mathcal{P}_i$,
$\mathcal{P}_j$, $\mathcal{P}_k$, representing vertices $s_i$,
$s_j$, $s_k$, respectively, then the path corresponding to vertex
$c_j$ in the clique has to intersect $\mathcal{P}_i$ and
$\mathcal{P}_k$ avoiding $\mathcal{P}_j$, and so it needs at least
three bends, a contradiction. So, we may assume that this
situation does not occur. Consider two opposite corners of the
rectangle. If they are not covered by two different paths, then
there are at least three paths whose union is contained in two
adjacent sides of the rectangle, a contradiction. So, we may
assume that the four corners are covered by paths, corresponding
to (at most) four vertices of the stable set. If no such path uses
two corners, then, from the remaining three paths, there are two
of them that are intervals contained in the same side or in
adjacent sides of the rectangle. But then, their union together
with one path on a corner are contained in two adjacent sides of
the rectangle, a contradiction. If exactly one of the paths uses
two corners (and one side of the rectangle), from the remaining
four paths, there are two of them that are intervals contained in
the same side of the rectangle, so as before we obtain a
contradiction. Finally, if there are two paths using two corners
each, and two sides of the rectangle, from the remaining five
paths, there are three of them that are intervals contained in the
same side of the rectangle, a contradiction. This completes the
proof.
\end{proof}

\begin{prop}\label{p:s7g}
The thick spider $S_{7}$ is not a $B_1$-EPG graph.
\end{prop}

\begin{proof}
By contradiction, suppose that $S_7$ admits a $B_1$-EPG
representation. Let us consider the path $\mathcal{P}_c$
corresponding to a vertex $c$ of the clique and the paths
corresponding to its $6$ neighbors in the stable set $S$. The path
$\mathcal{P}_c$ has edges on at most two lines (one row and one
column) of the grid. Thus, it intersects at least $3$ paths
corresponding to neighbors in $S$ on a same line $x$. Without loss
of generality, we may assume that $x$ corresponds to a row of the
grid. Consider now the $3$ intervals on row $x$ belonging to the 3
paths mentioned above. Since the corresponding vertices are
pairwise non adjacent, they admit some order on row $x$, say
$\mathcal{I}_i$, $\mathcal{I}_j$, $\mathcal{I}_k$, corresponding
to vertices $s_i, s_j,s_k$, respectively. Having at most one bend
each, the paths corresponding to $s_i$ and $s_k$ do not have edges
in a common column (and they do not have edges in a row different
from $x$, or more than one interval on $x$). So there is no way
for the path $\mathcal{P}_{c_j}$ corresponding to the vertex $c_j$
of the clique that is not adjacent to $s_j$ of avoiding the
interval $\mathcal{I}_j$ while intersecting the paths
corresponding to $s_i$ and $s_k$ using only one bend.
\end{proof}


\noindent \textbf{Acknowledgement.} We would like to thank Marty
Golumbic whose lectures about EPG graphs have motivated this
research, and to Jayme Szwarcfiter, who posed some of these
questions at LAGOS'09.


\end{document}